\documentclass[11pt,a4paper]{amsart}
\usepackage{graphics}
\usepackage{amscd}
\usepackage{comment}
\usepackage{oldgerm}
\usepackage{amsmath}
\usepackage{amssymb}
\usepackage{amsthm}
\usepackage{color}
\usepackage{bm}

\usepackage[all]{xy}
\usepackage[foot]{amsaddr}

\makeatletter
\@addtoreset{equation}{section}
\makeatother

\theoremstyle{definition}
\newtheorem{defn}[equation]{Definition}
\theoremstyle{plain}
\newtheorem{thm}[equation]{Theorem}

\newtheorem{prop}[equation]{Proposition}
\newtheorem{fact}[equation]{Fact}
\newtheorem{cor}[equation]{Corollary}
\newtheorem{lem}[equation]{Lemma}

\theoremstyle{remark}
\newtheorem{rem}[equation]{Remark}
\newtheorem{ex}[equation]{Example}

\newcommand{\Z}{\mathbb{Z}}
\newcommand{\Q}{\mathbb{Q}}
\newcommand{\N}{\mathbb{N}}
\newcommand{\R}{\mathbb{R}}
\newcommand{\C}{\mathbb{C}}
\newcommand{\B}{\mathbb{B}}

\newcommand{\del}{\partial}

\begin{document}

\title[A lattice version of the Atiyah-Singer index theorem]{A lattice version of the Atiyah-Singer index theorem}
\author[M. Yamashita]{Mayuko Yamashita}
\address{Research Institute for Mathematical Sciences, Kyoto University, 
606-8502, Kyoto, Japan}
\email{mayuko@kurims.kyoto-u.ac.jp}
\subjclass[]{}
\maketitle

\begin{abstract}
   We formulate and prove a lattice version of the Atiyah-Singer index theorem. 
   The main theorem gives a $K$-theoretic formula for an index-type invariant of operators on lattice approximations of closed integral affine manifolds. 
   We apply the main theorem to an index problem of Wilson-Dirac operators in lattice gauge theory.  
\end{abstract}

\tableofcontents

\section{Introduction}
In this paper we formulate and prove a lattice version of the Atiyah-Singer index theorem. 
Given a closed integral affine manifold, the main theorem gives a $K$-theoretic formula for an index-type invariant of operators on the lattice approximation of the manifold. 
This work is motivated from lattice gauge theory. 
We apply the main theorem to the index problem of Wilson-Dirac operator in lattice gauge theory, and prove relations between certain index-type invariants of Wilson-Dirac operators with the Fredholm index of twisted spin Dirac operators in the continuum limit.

First, let me explain the motivation from lattice gauge theory. 
In lattice gauge theory, manifolds, typically the $n$-dimensional torus $B:= T^n = (\R /\Z)^n$, are approximated by the set of level-$k$ lattice points $B_k := (\frac{1}{k}\Z /\Z)^n$. 
When they are interested in a differential operator $D^{\mathrm{conti}}$ on $B$, they construct its lattice counterparts $\{D^{\mathrm{lat}}_k\}_{k \in \N}$ on $B_k$'s, which is a family of operators on finite dimensional Hilbert spaces. 
One expects to recover information of the continuum operator $D^{\mathrm{conti}}$ from information of $\{D^{\mathrm{lat}}_k\}_{k \in \N}$. 
In this paper, we are interested in the Fredholm indices of elliptic operators on $B$, which describes the anomaly in physics. 

The typical setting is the following. 
Let $B= T^n$ with $n$ even, and $D^{\mathrm{conti}} \colon L^2(B; S \otimes F) \to L^2(B; S \otimes F)$ be the spin Dirac operator twisted by a hermitian vector bundle $F$ with a unitary connection. 
We are interested in its Fredholm index, $\mathrm{Ind}(D^{\mathrm{conti}})$. 
The first problem is how to construct a family of lattice operators $\{D^{\mathrm{lat}}_k\}_{k \in \N}$ which remembers the index, and what kind of invariant we consider for this family. 
This question is highly nontrivial; it turns out that the naive approximation does not work. 
Moreover, for example, the Fredholm indices of operators on finite dimensional vector spaces are not interesting. 
For this problem, one answer known in lattice gauge theory is to use the operators called the {\it Wilson-Dirac operators} $\{D^{\mathrm{lat}}_{ k} + \gamma W_k\}_k$, self-adjoint operators acting on $l^2(B_k; (S\otimes F)|_{B_k})$, and to count the dimensions of their positive eigenspaces. 
The relation between the spectrum of Wilson-Dirac operators and the Fredholm index of the continuum Dirac operator is predicted physically by Hasenfratz, Laliena and Niedermayer \cite{HLN1998}, and verified mathematically by Adams \cite{Adams2001} (there have been many related works, for example see \cite{IIY1987}, \cite{KY1999} and \cite{Suzuki1999}). 
Adams \cite{Adams2001} showed that (the author works in the case $n = 4$, but the method extends to arbitrary positive even integer $n$), for $m \in \R \setminus \{0, 2, 4, \cdots, 2n\}$, we have
\begin{align}\label{eq_intro_DWprob}
  \mathrm{rank} \left(E_{>0}\left(D^{\mathrm{lat}}_{k} + \gamma(W_k + mk)\right)\right) - \frac{1}{2} \dim l^2(B_k; (S \otimes E)|_{B_k}) \xrightarrow{k \to \infty}
  I_n(m)\mathrm{Ind}(D^{\mathrm{conti}}). 
\end{align}
Here the integer $I_n(m) \in \Z$ is defined in Definition \ref{def_coeff}; in particular we have $I_n(m) = 1$ for $0<m<2$. 
The first term of \eqref{eq_intro_DWprob} is the dimension of positive eigenspaces of the operator $D^{\mathrm{lat}}_{W, k} + m\gamma$, where $\gamma$ is the $\Z_2$-grading operator on $S\otimes F$. 
The proof uses analysis of the local index density, known as Fujikawa's method. 

This work started from the following question: {\it Can we understand the convergence \eqref{eq_intro_DWprob} conceptually and topologically?}
Recall that, on the continuum side, we know that the Fredholm index is a topological quantity, by the celebrated Atiyah-Singer index theorem \cite{AS1968}. 
\begin{thm}[{The Atiyah-Singer index theorem, \cite{AS1968}}]\label{thm_AS}
Given a closed manifold $M$ and an elliptic pseudodifferential operator $D^{\mathrm{conti}}$ on $M$, we have
\begin{align*}
      \mathrm{Ind}(D^{\mathrm{conti}})
      &= \pi_![\sigma(D^{\mathrm{conti}})]\\
      &= \int_{T^*M}ch(\sigma(D^{\mathrm{conti}}))td(\omega). 
\end{align*}
Here $[\sigma(D^{\mathrm{conti}})] \in K^0(T^*M)$ is the principal symbol class of $D^{\mathrm{conti}}$, and $\pi_! \colon K^0(T^*M) \to K^0(pt)$ is the spin$^c$-pushforward map with respect to the canonical symplectic structure $\omega$ on $T^*M$.  
\end{thm}
This leads us to the following problem: {\it Can we find a corresponding topological formula for the index-type invariant (e.g., the one appearing in \eqref{eq_intro_DWprob}), for operators on lattices?}
Such a theorem should be a lattice counterpart of the Atiyah-Singer index theorem. 
Then, the next problem is, {\it Apply the theorem to show the convergence \eqref{eq_intro_DWprob}}. 
This paper answers these problems. 

Now let me explain the main result. 
The setting is the following. 
Let $B$ be a closed integral affine manifold (for example $T^n$). 
The set of level-$k$ lattice points on $B$ (i.e., $\frac{1}{k}\Z^n$ in the local integral affine coordinates) is denoted by $B_k$. 
Write $\Lambda^* \subset T^*B$ the associated lattice subbundle. 
In this setting, our result computes the behavior of dimensions of positive eigenspaces for a certain class of families of self-adjoint operators on $\{B_k\}_k$, in terms of the $K$-theory class of their ``lattice version of symbols'', which is a function on the torus bundle $T^*B / (2\pi \Lambda^*)$.  

This ``lattice version of correspondence between operators and symbols'' is the one constructed in the previous paper of the author \cite{Yamashita2020}. 
Applied to our setting of the Lagrangian torus bundle $ T^*B / (2\pi \Lambda^*) \to B$, it produces a family of linear maps $\{\phi^k\}_k$, 
\begin{align*}
    \phi^k \colon C^\infty(T^*B / (2\pi \Lambda^*)) \to \mathrm{End}(l^2(B_k)). 
\end{align*}
This gives a strict deformation quantization of $X$, which we call the {\it Bohr-Sommerfeld deformation quantization} in this paper. 
This construction is an analogue of symbol-operator correspondence, as explained in \cite{Yamashita2020} and also recalled in subsection \ref{subsubsec_DQ_construction} below. 
Given an element $f \in C^\infty(T^*B / (2\pi \Lambda^*))$, the family of operators $\{\phi^k(f)\}_k$ on $\{B_k\}_k$ should be regarded as the operator realization of $f$, and the function $f$ is regarded as the lattice version of symbols of $\{\phi^k(f)\}_k$. 
These maps extends to matrix algebras canonically, and we continue to use the same notations. 

The lattice version of the Atiyah-Singer index theorem, our main theorem Theorem \ref{thm_lat_ind_inv.y}, is the following. 
Given an invertible and self-adjoint element $f \in M_N(C^\infty(T^*B / (2\pi \Lambda^*)))$, the element $(f(f^*f)^{-1/2}+1)/2$ is a projection\footnote{Here $M_N(A) := A \otimes M_N(\C)$ denotes the $N \times N$-matrix algebra for a $\C$-algebra $A$. }.  
Let us denote the corresponding $K^0$-theory class by $[f] \in K^0(T^*B / (2\pi \Lambda^*))$ (See Subsection \ref{subsec_Ktheory} for our conventions on $K$-theory). 

\theoremstyle{plain}
\newtheorem*{intro1}{Theorem \ref{thm_lat_ind_inv.y}}
\begin{intro1}[The lattice index theorem]
Fix a positive integer $N$. 
Suppose we are given an invertible self-adjoint element $f \in M_N(C^\infty(T^*B / (2\pi \Lambda^*)))$. 
Then there exists a positive integer $K$ such that, for all integer $k > K$, we have
\begin{align*}
  \mathrm{rank} \left(E_{>0}\left(\phi^k(f)\right)\right) 
  &= \pi_{!}\left([L]^{\otimes k}\otimes [f]\right)\\
  &=  (2\pi\sqrt{-1})^{-\dim B}\int_{T^*B/(2\pi \Lambda^*)} ch(f)td(\omega)e^{\sqrt{-1}k\omega}. 
\end{align*}
Here $\pi_! \colon K^0(T^*B/(2\pi \Lambda^*)) \to K^0(pt)$ is the spin$^c$-pushforward map with respect to the canonical symplectic structure $\omega$ on $T^*B/(2\pi \Lambda^*)$, and 
$[L]$ is the class of prequantum line bundle of $T^*B/(2\pi \Lambda^*)$ which is used in the construction of $\phi^k$. 
\end{intro1}

The main idea for the proof of Theorem \ref{thm_lat_ind_inv.y} is to apply the algebraic index theorem by Nest and Tsygan \cite{NestTsygan1995a} to the Bohr-Sommerfeld deformation quantization. 
Recall that, on the continuum side, deformation quantization and the index theorem are deeply related. 
Given a manifold $M$, the algebra of pseudodifferential operators on $M$ gives a deformation quantization for $T^*M$. 
As skeched in the introduction of \cite{NestTsygan1995a}, the Atiyah-Singer index theorem essentially (though not directly) follows from the algebraic index theorem applied to this deformation quantization. 
Our proof for Theorem \ref{thm_lat_ind_inv.y} is the lattice analogue of this picture. 
The Bohr-Sommerfeld deformation quantization for $T^*B /(2\pi \Lambda^*)$, which is a {\it strict} deformation quantization, induces a {\it formal} deformation quantization (in fact this is simply the standard Moyal-Weyl star product). 
After checking that we are in the appropriate setting, the proof is a direct application of the algebraic index theorem.  

This paper is organized as follows. 
In Section \ref{sec_preliminaries}, we recall necessary results about the Bohr-Sommerfeld deformation quantization from \cite{Yamashita2020}, and give a brief review of the algebraic index theorem \cite{NestTsygan1995a} and basics on topological $K$-theory. 
In Section \ref{sec_main}, we prove our main result, Theorem \ref{thm_lat_ind_inv.y}. 
In Section \ref{sec_application}, we apply the main theorem to the index problem in lattice gauge theory.
In particular we prove the above convergence \eqref{eq_intro_DWprob} in Theorem \ref{thm_application}.

\subsection{Conventions and notations}
\begin{itemize}
 \item Given a fiber bundle $\mu \colon X \to B$ and a point $b \in B$, we write $X_b := \mu^{-1}(b)$.
    \item For a self-adjoint operator $D$ on a separable Hilbert space and a real number $\lambda$, we denote by $E_{>\lambda}(D)$ the spectral projection of $D$ corresponding to the interval $(\lambda, \infty)$. 
    \item For a Hilbert space $H$, $\mathbb{B}(H)$ denotes the $C^*$-algebra of bounded operators on $H$. 
    \item Given a space $X$ and a vector space $V$, we denote the trivial vector bundle over $X$ with fiber $V$ by $\underline{V} := X \times V$. 
\end{itemize}

\section{Preliminaries}\label{sec_preliminaries}
\subsection{The Bohr-Sommerfeld deformation quantization for cotangent torus bundles}\label{subsec_preliminaries_BS}

First we introduce the main object of this paper. 
\begin{defn}[Integral / tropical affine manifolds]\label{def_affine}
An {\it integral affine structure} (resp. {\it tropical affine structure}) on a smooth $n$-dimensional manifold $B$ is a local coordinate system whose transition functions are given by integral (resp. tropical) affine transformations, i.e., elements in $GL_n(\Z)\ltimes \Z^n$ (resp. $GL_n(\Z)\ltimes \R^n$). 
A manifold equipped with an integral (resp. tropical) affine structure is called an {\it integral affine manifold} (resp. {\it tropical affine manifold}).
\end{defn}
\begin{rem}
There are some variants in the conventions. 
In some literatures, tropical affine manifolds in Definition \ref{def_affine} are called integral affine manifolds. 
\end{rem}

For each $k \in \Z_{> 0}$, integral affine transformations preserve the level-$k$ lattice points $\frac{1}{k}\Z^n$ of $\R^n$. 
Thus, given an integral affine structure on $B^n$, the set of level-$k$ lattice points in each local coordinate glues together to give the set of {\it level-$k$ lattice points} of $B$, which we denote by $B_k \subset B$. 

A tropical affine structure on $B^n$ induces a lattice subbundle $\Lambda$ of its tangent bundle $TB$ (i.e., $\Lambda$ is a fiber bundle over $B$ and $\Lambda_b$ is a subgroup of $T_bB$ isomorphic to $\Z^n$ for all $b \in B$). 
We denote this pair by $(B, \Lambda)$ throughout this paper. 

Given an integral affine manifold $B$, by the results in \cite{Yamashita2020} we get a family of linear maps $\{\phi^k\}_k$, 
\begin{align*}
    \phi^k \colon C^\infty(T^*B / (2\pi \Lambda^*)) \to \mathrm{End}(l^2(B_k)), 
\end{align*}
which plays a role of ``symbol-operator correspondence'' in this paper. 
In the rest of this subsection, we recall the necessary result of \cite{Yamashita2020}. 

\subsubsection{The construction}\label{subsubsec_DQ_construction}
The definition of strict deformation quantizations we use is the following. 

\begin{defn}[Strict deformation quantizations]\label{def_DQ}
Given a symplectic manifold $(X, \omega)$, a {\it strict deformation quantization} consists of the following data. 
\begin{itemize}
    \item A sequence of Hilbert spaces $\{\mathcal{H}_k\}_{k \in \N}$. 
    \item A sequence $\{Q^k\}_{k \in \N}$ of adjoint-preserving linear maps $Q^k \colon C_c^\infty(X) \to \mathbb{B}(\mathcal{H}_k)$ so that for all $f, g \in C_c^\infty(X)$, we have
\begin{enumerate}
    \item $\|Q^k(f)\| \to \|f\|_{C^0}$ as $k \to \infty$, and
    \item $\|[Q^k(f), Q^k(g)] +\frac{\sqrt{-1}}{k}Q^k(\{f, g\})\| =O(\frac{1}{k^2})$ as $k \to \infty$. 
\end{enumerate}
\end{itemize}
\end{defn}

The general setting of \cite{Yamashita2020} is the following. 
Assume that we are given a symplectic manifold $(X, \omega)$ equipped with a prequantum line bundle $(L, \nabla^L)$, and also assume that we are given a proper Lagrangian fiber bundle structure $\mu \colon X \to B$ with connected fibers.
Here, 
\begin{defn}
Let $(X^{2n}, \omega)$ be a symplectic manifold of dimension $2n$. 
\begin{enumerate}
    \item A {\it prequantum line bundle} on $(X, \omega)$ is a hermitian line bundle with unitary connection $(L, \nabla^L)$ which satisfies $(\nabla^L)^2 = -\sqrt{-1}\omega$. 
    \item A regular fiber bundle structure $\mu \colon X^{2n} \to B^n$ is called a {\it Lagrangian fiber bundle} if all the fibers are Lagrangian. 
    It is called {\it proper} if all fibers are compact.
\end{enumerate}
\end{defn}

\begin{ex}
Let $(B^n, \Lambda)$ be an $n$-dimensional tropical affine manifold. 
Consider the cotangent torus bundle $T^*B / (2\pi \Lambda^*)$ over $B$, where $\Lambda^*$ denotes the dual lattice bundle to $\Lambda$. 
We equip $T^*B / (2\pi \Lambda^*)$ with the canonical symplectic structure induced from $T^*B$. 
Then the fiber bundle $\mu \colon T^*B / (2\pi \Lambda^*) \to B$ is a proper Lagrangian fiber bundle with fiber $(\R / (2\pi \Z))^n$. 
By the Arnold-Liouville theorem \cite{Arnold1989}, any proper Lagrangian fiber bundle is locally of this form. 
\end{ex}

Given a symplectic manifold, the exisitence of prequantum line bundle is equivalent to the condition $\omega /(2\pi)\in H^2(X; \Z)$. 
In this settings, the author constructed a strict deformation quantization for $(X, \omega)$. 
In this paper we call it the {\it Bohr-Sommerfeld deformation quantization}. 

The representation spaces $\{\mathcal{H}_k\}_k$, called the {\it quantum Hilbert spaces}, of the strict deformation quantization defined in \cite{Yamashita2020} is the ones given by the geometric quantization associated to the real polarization $\mu$, as we now explain. 
A proper Lagrangian fiber bundle $\mu \colon X \to B$ together with a prequantum line bundle $(L, \nabla^L)$, indues an integral affine structure on the base space $B$, for which the set of level-$k$ lattice points $B_k \subset B$ is given by the set of {\it $k$-Bohr-Sommerfeld points} (see Lemma \ref{lem_aff_Lag}). 
The quantum Hilbert spaces of our deformation quantization is given by a direct sum of one-dimensional Hilbert spaces, associated to each $k$-Bohr-Sommerfeld points, as follows. 

\begin{defn}
Assume we are given a prequantized symplectic manifold $(X, \omega, L, \nabla)$ equipped with a proper Lagrangian fiber bundle structure $\mu \colon X \to B$ with connected fibers. 
Let $k$ be a positive integer. 
\begin{enumerate}
    \item A point $b \in B$ is called a $k$-{\it Bohr-Sommerfeld point} if the space of parallel sections of $(L^k, \nabla^k)|_{X_b}$ is nontrivial.
    \item For each $k$, let $B_k \subset B$ denote the set of $k$-Bohr-Sommerfeld points.
    We define the quantum Hilbert space of level $k$ by
    \begin{align*}
        \mathcal{H}_k = \oplus_{b \in B_k} H^0(X_b; L^k\otimes |\Lambda|^{1/2}X_b),
    \end{align*}
    where $|\Lambda|^{1/2}X_b = |\Lambda|^{1/2} (\ker d\mu)^*|_{X_b}$ is the vertical half-density bundle, equipped with the canonical flat connection, and $H^0(X_b; L^k\otimes |\Lambda|^{1/2}X_b)$ is the one-dimensional Hilbert space of parallel sections of $L^k|_{X_b}\otimes |\Lambda|^{1/2}X_b$ over $X_b$ for each $b \in B_k$. 
\end{enumerate}
\end{defn}

\begin{ex}\label{ex_preq}
For the case $(X, \omega) = (\R^n \times (\R /(2 \pi \Z))^n, {}^t\! dx \wedge d\theta)$ with the projection $\mu \colon X \to \R^n$, we can set $(L, \nabla^L) = (\underline{\C}, d -\sqrt{-1}{}^t\! x  d\theta))$. 
Then we have $B_k = \frac{1}{k}\Z^n$. 

The base $\R^n$ admits a $\Z^n$-action by translation. 
This action lifts to the above prequantum line bundle by 
\begin{align*}
    (x, \theta, v) \mapsto (x+m, \theta, e^{\sqrt{-1}\langle m, \theta \rangle} v), 
\end{align*}
preserving the connection. 
So we get the induced prequantum line bundle on $(\R /\Z)^n \times (\R /(2 \pi \Z))^n$. 
In this case, the set of $k$-Bohr-Sommerfeld point is given by $B_k = (\frac{1}{k}\Z / \Z)^n\subset (\R /\Z)^n$. 
\end{ex}

We have the following relations between integral affine manifolds and Lagrangian fiber bundles. 
\begin{lem}\label{lem_aff_Lag}
\begin{enumerate}
    \item Let $(X^{2n}, \omega)$ be a symplectic manifold of dimension $2n$. 
    A proper Lagrangian fiber bundle structure $\mu \colon X^{2n} \to B^n$ with connected fibers canonically induces a tropical affine structure on $B$.  
    \item The data $(X, \omega, \mu, B)$ as above together with a prequantum line bundle $(L, \nabla)$ of $(X, \omega)$ canonically induce an integral affine structure on $B$ whose level-$k$ lattice points coincide with $k$-Bohr-Sommerfeld points. 
    \item Conversely, given an integral affine manifold $B$, there exists a prequantum line bundle $(L, \nabla)$ of $T^*B / (2\pi\Lambda^*)$ equipped with the canonical symplectic structure, whose $k$-Bohr-Sommerfeld points with respect to the Lagrangian fiber bundle structure $T^*B / (2\pi\Lambda^*) \to B$ coincide with level-$k$ lattice points of $B$. 
    Moreover, $(L, \nabla)$ can be taken so that there is a canonical isomorphism of Hilbert spaces, 
\begin{align}\label{eq_isom_hilb_l2}
    \mathcal{H}_k \simeq l^2(B_k),
\end{align}
for each $k$. 
\end{enumerate}
\end{lem}

\begin{proof}
(1) is the direct consequence of the Arnold-Liouville theorem \cite{Arnold1989}. 
Namely, any proper Lagrangian fiber bundle with connected fibers is locally isomorphic to the standard example $(\R^n \times (\R /(2 \pi \Z))^n, {}^t\! dx \wedge d\theta)\to \R^n$. 
Any local automorphism of this local coordinate $(x, \theta)$ induces a tropical affine transformation on the base coordinate $x$. 
Thus, by choosing arbitrary local isomorphisms with the standard example, we get the desired tropical affine structure on $B$, which is independent of the choice of local isomorphisms. 

For (2), use the fact that any prequantum line bundle on a proper Lagrangian fiber bundle of dimension $2n$ with connected fibers is locally isomorphic to the one in Example \ref{ex_preq} (see the proof of \cite[Lemma 2.8]{Yamashita2020}). 
Any local automorphism of the standard example induces an integral affine transformation on the base $\R^n$, and the Bohr-Sommerfeld points coincide with the level-$k$ lattice points in each local coordinate.  
Thus, by choosing arbitrary local isomorphisms with the standard example, we get the desired integral affine structure on $B$, which is independent of the choice of local isomorphisms. 

For (3), assume we are given an $n$-dimensional integral affine manifold $(B^n, \Lambda)$ with a local coordinate system $(\{U_i\}_i, \{\varphi_{ij}\}_{i, j})$, with $U_i \subset \R^n$ and $\varphi_{ij} \in GL_n(\Z)\ltimes \Z^n$. 
On each $U_i$ we have $T^*U_i / (2\pi\Lambda^*) = U_i \times (\R /(2 \pi \Z))^n$ and the canonical symplectic structure is the standard one ${}^t\! dx \wedge d\theta$. 
Let us equip $T^*U_i / (2\pi\Lambda^*)$ with the standard prequantum line bundle $(L_i := \underline{\C}, \nabla_i =d -\sqrt{-1}{}^t\! x  d\theta)$ in Example \ref{ex_preq}. 
The coordinate transformation $\varphi_{ij} = (x \mapsto Ax + b)$ lifts to the isomorphism $\tilde{\varphi}_{ij} \colon (L_j, \nabla_j) \simeq (L_i, \nabla_i)$ by
\begin{align*}
    (x, \theta, v) \mapsto (Ax + b, {}^t\!A^{-1}\theta, e^{\sqrt{-1}\langle b, {}^t\!A^{-1}\theta\rangle} v). 
\end{align*}
It is easy to see that $\tilde{\varphi}_{ij} \circ \tilde{\varphi}_{jk}=\tilde{\varphi}_{ik}$, so the prequantum line bundles $\{(L_i, \nabla_i)\}_{i, j}$ glue together to give a line bundle $(L, \nabla)$ on $T^*B / (2\pi\Lambda^*)$. 
By construction the set of $k$-Bohr-Sommerfeld points of it coincides with the set of level-$k$ lattice points of $B$. 
For the isomorphism \eqref{eq_isom_hilb_l2}, 
note that $(L, \nabla)$ constructed above is equipped with a trivialization on the zero section $X_0 \subset X$ by the construction. 
This trivialization gives the canonical orthonormal basis $\{\psi^k_b\}_{b \in B_k}$ of $\mathcal{H}_k$ by requiring that each $\psi^k_b \in H^0(X_b; L^k\otimes |\Lambda|^{1/2}X_b)$ takes the positive real value at the point $X_0 \cap X_b$.
This gives the canonical isomorphism \eqref{eq_isom_hilb_l2}.   
\end{proof}

In the rest of this subsection, we assume that $X$ is of the form $X = T^*B / (2\pi \Lambda^*)$ for an integral affine manifold $B$.  
Restricted to this setting, the construction of the strict deformation quantization simplifies, and described as follows. 
Take a prequantum line bundle $(L, \nabla)$ for $X$ such that the set of Bohr-Sommerfeld points coincides with the lattice points of $B$. 
The existence of such $(L, \nabla)$ is guaranteed by Lemma \ref{lem_aff_Lag} (3). 

\begin{rem}
Actually, in the constructions below, as well as in our main theorem, we do not need to assume that $(L, \nabla^L)$ is equipped with an isomorphism \eqref{eq_isom_hilb_l2}. 
However, when we apply our result to problems on operators on lattices, as in Section \ref{sec_application} below, we start from operators on $l^2(B_k)$. 
In such a situation, Lemma \ref{lem_aff_Lag} (3) guarantees the existence of an appropriate choice of $(L, \nabla^L)$. 
\end{rem}

In \cite[Definition 3.2]{Yamashita2020}, we constructed linear maps
\begin{align*}
    \phi^k_{H, \mathcal{U}} \colon C_c^\infty(X) \to \mathbb{B}(\mathcal{H}_k), 
\end{align*}
and showed that indeed this gives a strict deformation quantization (\cite[Theorem 3.32]{Yamashita2020}). 
Here, the additional datum $(H, \mathcal{U})$ were necessary: 
$H \subset TX$ is a choice of horizontal distribution with respect to $\mu$, and $\mathcal{U}$ is an open covering of $B$, which satisfy some conditions ((H) and (U) in \cite[subsection 3.2]{Yamashita2020}). 

In our setting here, we have a canonical choice of $H$, coming from the canonical splitting $TX = \mu^*TB \oplus \mu^* T^*B$. 
In this paper we always use this splitting to define the strict deformation map, so we omit the reference to $H$ in the notation. 
On the other hand, the choice of $\mathcal{U}$ is only technical (just needed to patch local constructions together), and the different choice of $\mathcal{U}$ yields essentially the same deformation quantization (\cite[Proposition 3.35]{Yamashita2020}). 
Since our result in this paper does not depend on this choice, we fix such an open covering $\mathcal{U}$ arbitrarily first, and also omit from the notation
\footnote{
The essential points of the condition (U) imposed on the open covering $\mathcal{U}$ is that, each element $U \in \mathcal{U}$ admits an integral affine open embedding into $\R^n$ whose image is relatively compact and convex, and for each pair of elements $U, V \in \mathcal{U}$, the image of the affine embedding in $\R^n$ of their intersection $U \cap V$ is also convex (in particular connected).  
This condition allows us to, given two points $b, c \in B$ which are {\it close} (i.e., contained in some common element in $\mathcal{U}$), find a unique affine linear path from $b$ to $c$ contained in some element in $\mathcal{U}$. 
}. 

We regard the quantization maps $\phi^k \colon C^\infty_c(T^*B/(2\pi \Lambda^*)) \to \B(\mathcal{H}_k)$ in our setting as {\it a lattice version of the correspondence between symbols and operators}.
The idea of this construction is the fiberwise Fourier expansion of functions on the cotangent torus bundle. 
We recall the rigorous definition first, and explain this idea after that. 

Given a path $\gamma$ in $B$ from $b\in B$ to $c \in B$, the restriction of the cotangent lattice bundle to $\gamma$, $\Lambda^*|_{\gamma}$, is trivial. 
So we get the parallel transform 
\begin{align}\label{eq_para.y}
    T_\gamma \colon X_b \xrightarrow{\simeq} X_c. 
\end{align}
Also the connection $\nabla^L$ on $L$ and the canonical flat connection on $|\Lambda|^{1/2}(\ker d\mu)^*$ gives the parallel transform
\begin{align*}
    T_\gamma \colon L^k|_{X_b} \otimes |\Lambda|^{1/2}X_b \to L^k|_{X_c} \otimes |\Lambda|^{1/2}X_c 
\end{align*}
which covers \eqref{eq_para.y}. 
We use the same notation for the parallel transform. 
This allows us to define a pairing between sections $\xi_b^k \in C^\infty(X_b; L^k \otimes |\Lambda|^{1/2}X_b)$ and $\xi_c^k \in C^\infty(X_c; L^k \otimes |\Lambda|^{1/2}X_c)$, denoted by $\langle \xi_b^k, \xi_c^k \rangle_\gamma$. 

We say that two points $b, c \in B$ are {\it close} if there exists an element $U \in \mathcal{U}$ such that $b, c \in U$. 
For such $b, c \in B$, by the condition (U) imposed on $\mathcal{U}$ (see \cite[subsection 3.2]{Yamashita2020}) we can take the unique affine linear path $\gamma$ from $b$ to $c$ in $U$ and define, for sections $\xi_b^k \in C^\infty(X_b; L^k \otimes |\Lambda|^{1/2}X_b)$ and $\xi_c^k \in C^\infty(X_c; L^k \otimes |\Lambda|^{1/2}X_c)$,
\begin{align*}
    \langle \xi_b^k, \xi_c^k \rangle_\mathcal{U} := \langle \xi_b^k, \xi_c^k \rangle_\gamma. 
\end{align*}
This is well-defined by the condition (U) on $\mathcal{U}$. 

\begin{defn}[The Bohr-Sommerfeld deformation quantization, {\cite[Definition 3.22]{Yamashita2020}}]\label{def_BSquantization}
We define a sequence of adjoint-preserving linear maps $\phi^k \colon C_c^\infty(X) \to \mathbb{B}(\mathcal{H}_k)$ by the following formula. 
For $f \in C_c^\infty(X)$, we define the operator $\phi^k(f)$ by, for $c \in B_k$ and an element $\psi^k_c \in H^0(X_c; L^k\otimes |\Lambda|^{1/2}X_c) \subset \mathcal{H}_k$,  
\begin{align*}
\phi^k(f)(\psi^k_c) := \sum_{b \in B_k, b\mbox{ is close to }c} \langle \psi^k_b, f|_{X_{(b + c)/2}}\psi^k_c \rangle_{\mathcal{U}}  \cdot \psi^k_b, 
\end{align*}
where $\psi^k_b \in H^0(X_b; L^k\otimes |\Lambda|^{1/2}(X_b)) \subset \mathcal{H}_k$
is any element with $\|\psi^k_b\| = 1$. 
Here, we denote by $(b + c)/2 \in B$ the middle point between $b$ and $c$ with respect to the affine structure on an open set $U \in \mathcal{U}$ which contains both $b$ and $c$, and we regard $f|_{X_{(b + c)/2}} \in C^\infty(X_{(b+c)/2})$ as a function on $X_c$ using the parallel transform \eqref{eq_para.y} along the affine linear path between $(b + c)/2$ and $c$ in $U$. 

This construction gives a strict deformation quantization for $(X, \omega)$ (\cite[Theorem 3.32]{Yamashita2020}), and we call it the {\it Bohr-Sommerfeld deformation quantization}. 
\end{defn}

Now we explain that this definition is indeed the fiberwise Fourier expansion. 
Locally on an open subset $U \subset B^n$ which is small enough, we can choose an open embedding $U \hookrightarrow \R^n$ which preserves the integral affine structure, so from now on we explain in the case of $X=T^*\R^n /(2\pi \Lambda^*) = \R^n \times (\R/2\pi \Z)^n$. 

Equip $X$ with the prequantizing line bundle $(L = \underline{\C}, \nabla^L = d -\sqrt{-1} {}^t\!xd\theta)$. 
Up to parallel translation of the base $\R^n$, any choice of $(L, \nabla^L)$ is isomorphic to this canonical one (see the proof of \cite[Lemma 2.8]{Yamashita2020}). 

In this case we have $B_k = \frac{1}{k}\Z^n$. 
The canonical orthonormal basis $\{\psi_b^k\}_{b \in B_k}$ for $\mathcal{H}_k$ in the proof of Lemma \ref{lem_aff_Lag} (3) is given by
\begin{align*}
    \psi_b^k := e^{ \sqrt{-1} k\langle b, \theta \rangle}(2\pi)^{-n/2}\sqrt{d\theta} \in \mathcal{H}_k. 
\end{align*}

Assume we are given a function $f \in C_c^\infty(X)$. 
Using the above basis of $\mathcal{H}_k$, the operator $\phi^k(f)$ is identified by a $B_k \times B_k$-matrix $\{K_f(b, c)\}_{b, c \in B_k}$. 
Matrix elements $K_f(b, c)$ for $b, c \in B_k$ is given as follows. 
\begin{align}\label{eq_mat_elem_model.y}
    K_f(b, c) := (2\pi)^{-n}\int_{(\R/2\pi \Z)^n} e^{- \sqrt{-1} k \langle b - c, \theta \rangle} f((b+c)/2, \theta) d\theta. 
\end{align}
In other words, $K_f(b, c)$ is given by the $k(b-c)$-th coefficient in the Fourier expansion of $f((b+c)/2, \theta)$.

\begin{ex}
Assume $f \in C_c^\infty(X)$ is a pullback of a function $f_0 \in C_c^\infty(\R^n)$ on the base $\R^n$, i.e., $f$ does not depend on $\theta$. 
Then $\phi^k(f)$ is just the diagonal multiplication operator by the value of $f_0$ at each point on $B_k$, 
\begin{align*}
    K_f(b, c) = \begin{cases}
    f_0(c) & \mbox{ if } b = c, \\
    0 & \mbox{ otherwise. }
    \end{cases}
\end{align*}
\end{ex}
\begin{ex}\label{ex_concentration}
Assume $f$ can be expressed as $f(x, \theta) = f_m(x) e^{\sqrt{-1}\langle m, \theta \rangle}$ for some $m \in \Z^n$ and a function $f_m \in C^\infty_c(\R^n)$.
Then we have
\begin{align*}
    K_f(b, c) = 
    \begin{cases}
    f_m\left(c + m / (2k)\right) & \mbox{ if } b = c + m/k, \\
    0 & \mbox{ otherwise. }
    \end{cases}
\end{align*}
We see that the function $e^{\sqrt{-1}\langle m, \theta \rangle}$ plays the role of ``$m/k$-shift'', and if we let $k \to \infty$, the matrix elements of this operator concentrate to the diagonal. 
\end{ex}
In fact, the ``concentration to the diagonal'' of the matrix elements of the operator $\phi^k(f)$ as $k\to \infty$ seen in the above examples holds in general, because the Fourier coefficients of smooth function on $(\R/(2\pi\Z))^n$ is rapidly decreasing. 
Basically, this is why we can extend this construction to general Lagrangian fiber bundles by patching the local construction together by $\mathcal{U}$, and the different choice of $\mathcal{U}$ yields essentially the same deformation quantization. 

\subsubsection{The associated star product}

In general, given a strict deformation quantization in the sense of Definition \ref{def_DQ}, one expects that it induces a formal deformation quantization, i.e., a an associative product $\star$ on $C^\infty(X)[[\hbar]]$ which satisfies
\begin{align*}
    f \star 1 &= 1 \star f = f ,\\ 
    f \star g &= fg + O(\hbar), \\
    f \star g - g \star f &= \hbar \{f, g\} + O(\hbar^2), 
\end{align*}
for all $f, g \in C^\infty(X)$. 
We also assume that each coefficients of $\hbar^i$ in the star product $f\star g$ is a differential expression of $f$ and $g$. 
This is possible if we can expand the composition of operators the form $Q^k(f)Q^k(g)$ in a power series of $k^{-1}$, satisfying appropriate conditions. 

In our case (note that we are assuming $X = T^*B / (2\pi \Lambda^*)$), $X$ has the canonical flat torsion-free symplectic connection, so we have the caononical formal deformation quantization of $X$, called the {\it Moyal-Weyl star product} $\star_{MY}$ (see for example \cite{Weinstein1995}). 
Bohr-Sommerfeld deformation quantization indeed induces the Moyal-Weyl star product, i.e., 
informally, we have
\begin{align*}
    \phi^k(f \star_{MY} g) = \phi^k(f) \phi^k(g) \mod O(k^{-\infty}). 
\end{align*}
More precisely the statement is the following. 
Let us denote the standard Moyal-Weyl star product by $\star_{MY}$, and each coefficient by $\mathcal{C}_j$, i.e., 
\begin{align*}
    f \star_{MY} g = \sum_{j = 0}^\infty \hbar^j \mathcal{C}_j(f, g). 
\end{align*}
\begin{prop}[{\cite[Theorem 4.3]{Yamashita2020}}]\label{prop_star_is_comp}
Assume that $X$ is of the form $X = T^*B / (2\pi \Lambda^*)$ for an integral affine manifold $B$, and $X$ is equipped with a prequantum line bundle $(L, \nabla^L)$. 
Then for all $f , g \in C_c^\infty(X)$ and $l \in \N$, 
\begin{align*}
        \left\| \phi^k(f) \phi^k(g) -
        \sum_{j = 0}^l \left( \frac{-\sqrt{-1}}{k}\right)^j \phi^k\left(\mathcal{C}_j(f, g)\right)
        \right\| =O\left(\frac{1}{k^{l+1}}\right)
\end{align*}
as $k \to \infty$. 
\end{prop}

\subsection{A review of the algebraic index theorem}
In this subsection, we recall the algebraic index theorem by Nest and Tsygan \cite{NestTsygan1995a}, which is the main tool for our proof of the main theorem. 
Here we focus on the case of closed manifolds. 

Let $(X^{2n}, \omega)$ be a closed symplectic manifold of dimension $2n$. 
Suppose we are given a formal deformation quantization $\star$ for $(X, \omega)$. 
Let us denote by $\theta \in H^2(X; \C[[\hbar]])$ the characteristic class of this deformation quantization (\cite[Section 5]{NestTsygan1995a}, \cite{Fedosov1994}). 
Note that we have $\theta = \omega + O(\hbar)$. 

A {\it trace functional}  for $\star$ is a $\C[[\hbar]]$-linear map $\tau \colon C^\infty(X)[[\hbar]] \to \C[\hbar^{-1}, \hbar]]$, which satisfies
\begin{align*}
    \tau(f\star g) = \tau(g \star f)
\end{align*}
for all $f, g \in C^\infty(X)$. 
Trace functionals always exist and they are unique up to multiplication of elements in $\C[\hbar^{-1}, \hbar]]$. 
There is a canonical choice of normalization (\cite[Section 1]{NestTsygan1995a}). 
We denote this trace functional by $\tau$. 
It extends to matrix algebras $M_N(C^\infty(X))$ canonically. 

\begin{rem}
This normalization is determined by the following condition. 
Given a star product $\star$ on $X$, we can find an open set $U \subset X$ small enough, so that there exists an open subset $U_0 \subset \R^{2n}$ with the standard symplectic form $\omega_0$, and an isomorphism $g_U \colon (C^\infty(U)[[\hbar]], \star) \simeq (C^\infty(U_0)[[\hbar]], \star_{MY})$. 
Then, for $f \in C^\infty_c(U)$, we require that
\begin{align}\label{eq_normalization}
\tau(f) = \hbar^{-n}(n!)^{-1}\int_{U_0}g_U(f)\omega_0^n. 
\end{align}
\end{rem}

\begin{rem}\label{rem_trace}
In particular, in our setting where $X = T^*B / (2\pi \Lambda^*)$ for an integral affine manifold $B$ and we are considering the standard Moyal-Weyl star product globally on $X$, the canonical trace functional is simply, 
\begin{align*}
  \tau(f) = \hbar^{-n}(n!)^{-1}\int_{X}f\omega^n  
\end{align*}
for any $f \in C^\infty(X)$. 
\end{rem}

In this situation, the algebraic index theorem by Nest and Tsygan \cite[Theorem 1.1.1]{NestTsygan1995a} states the following. 

\begin{fact}[{The algebraic index theorem, \cite[Theorem 1.1.1]{NestTsygan1995a}}]\label{fact_alg_ind_thm.y}
Fix a positive integer $N$. 
Suppose we are given an idempotent $e \in M_N(C^\infty(X))[[\hbar]]$ with respect to the star product $\star$. 
Let us write
\begin{align*}
    e = e_0 + \hbar e_1 + \hbar^2 e_2 + \cdots,  
\end{align*}
where $e_i \in M_N(C^\infty(X))$. 
Then we have
\begin{align*}
    \tau (e) = \int_X ch(e_0)td(\omega)e^{-c_1(\omega)/2}e^{\theta/ \hbar}. 
\end{align*}
Here, the Chern character $ch(e_0)\in \Omega^{\mathrm{even}}(X)$ of the idempotent $e_0 \in M_N(C^\infty(X))$ is defined by
\begin{align*}
    ch(e_0) := \sum_{m = 0}^n \frac{1}{m!}\mathrm{tr}(e_0(de_0)^{2m}). 
\end{align*}
The classes $td(\omega)$ and $c_1(\omega)$ are the characteristic classes of $X$ with respect to the almost complex structure compatible with $\omega$. 
\end{fact}

The Chern character $ch(e_0) \in \Omega^{\mathrm{even}}(X)$ is the Chern character form of the connection $e_0de_0$ of the vector bundle $e_0 \cdot \underline{\C^N}$. 

\subsection{A review on $K$-theory}\label{subsec_Ktheory}
In this subsection, we briefly review basics on topological $K$-theory necessary in this paper. 
There are many nice references for this topic. 
For example see \cite{AS1968}, \cite{LawsonMichelsohn1989} and \cite{Bla1998} for details. 

\subsubsection{Definition and conventions}
In this paper we always use {\it compactly supported $K$-theory} on locally compact topological spaces. 
For a compact space $X$, we define its $K^0$-group $K^0(X)$ to be the Grothendieck group associated to the abelian semigroup of isomorphism classes of complex vector bundles over $X$, with additive structure given by direct sum. 
So an element in $K^0(X)$ is represented as the formal difference $[E] - [F]$ of classes of two complex vector bundles $E$ and $F$ over $X$. 
A continuous map $f \colon X \to Y$ induces a homomorphism $f^* \colon K^0(Y) \to K^0(X)$ by the pullback. 
For possibly noncompact $X$, we define
\begin{align*}
    K^0(X) := \mathrm{ker}\left(K^0(X^+) \to K^0(pt)
    \right), 
\end{align*}
where $X^+$ is the one-point compactification of $X$ and the map is induced by the inclusion of a point.
This means that an element in this group is given by an isomorhism class of a pair of complex vector bundles over $X$ together with an isomorphism between them defined on the complement of a compact subset in $X$. 

Here, it is convenient to introduce $\Z_2$-graded vector bundles. 
In this picture, an element in $K^0(X)$ is given by a homotopy class $[E, \sigma]$, where $E = E_+ \oplus E_-$ is a $\Z_2$-graded complex vector bundle over $X$, and $\sigma \colon E_+ \to E_-$ is a homomorphism which is invertible outside a compact set. 
In particular for compact $X$, the element $[E, \sigma]$ in this picture corresponds to the difference class $[E_+] - [E_-]$. 

In this paper, we also represent classes in $K$-groups in an operator-algebraic way.
Let $X$ be a compact space. 
For a positive integer $N$, 
we say that a matrix-valued function $p \in M_N(C(X))$ is a {\it projection} if it satisfies $p^* = p$ and $p^2 = p$, i.e., its value at each point in $X$ is a self-adjoint idempotent matrix. 
A projection (or just an idempotent) $p \in M_N(C(X))$ determines a class $[p] := [p\underline{\C^N}] \in K^0(X)$. 

We also represent classes in $K$-groups by an invertible and self-adjoint element in $M_N(C(X))$ for compact $X$. 
Given such $u \in M_N(C(X))$, we define $[u] \in K^0(X)$ to be the element represented by the projection $\frac{u(u^*u)^{-1/2} + 1}{2} \in M_N(C(X))$\footnote{
This point of view is generalized and leads us to the notion of $K$-theory for operator algebras, where $K$-groups are defined in terms of projections and unitaries. 
See \cite{Bla1998} for details. 
}. 

To summarize, we represent classes in $K$-groups by vector bundles, projections and self-adjoint invertible elements. 
Which picture we are using will be clear from the context. 

\subsubsection{Bott periodicity and the $K$-theory push-forward}
We define 
\begin{align*}
    K^{-i} (X) := K^0(\R^i \times X) 
\end{align*}
for $i \in \Z_{\ge 0}$. 
We have the {\it Bott periodicity} isomorphism, 
\begin{align*}
    K^0(pt) \simeq K^0(\R^2), 
\end{align*}
which sends the generator $[\C] \in K^0(pt)$ to the {\it Bott element} $[S, \sigma] \in K^0(\R^2)$, 
where $S$ is the complex spinor space in two-dimension (which is $\Z_2$-graded) and $\sigma$ is the Clifford multiplication. 
This gives a functorial isomorphism 
$K^{-i}(X) \simeq K^{-i-2}(X)$ for any $X$. 
Regarding $i$ as an integer mod two, we also define $K^i$ for positive $i$. 

$K$-theory orientation is given by {\it spin$^c$-structure}. 
This means that, for a real vector bundle $V \to X$ of rank $r$ with a spin$^c$-structure, we have
the Thom isomorhism in $K$-theory, 
\begin{align}\label{eq_Thom_isom}
    K^{i+r}(V) \simeq K^i(X),  
\end{align}
which generalizes the Bott periodicity. 
We say that a smooth map $f \colon X \to Y$ between manifolds is {\it $K$-oriented} if the bundle $TX \oplus f^*TY$ is equipped with a spin$^c$-structure. 
For such a map, we have the {\it spin$^c$-pushforward map}, 
\begin{align*}
    f_! \colon K^i(X) \to K^{i + \dim Y - \dim X}(Y), 
\end{align*}
characterized by the following properties. 
\begin{itemize}
    \item It is functorial, i.e. $(f\circ g)_! = f_! \circ g_!$. 
    \item If $i \colon U \to X$ is an open embedding, $i_!$ coincides with the natural homomorphism associated to the map $X^+ \to U^+$. 
    \item If $\pi \colon V \to X$ is a real vector bundle with a spin$^c$-structure, and $i \colon X \to V$ is the inclusion of the zero section, $\pi_!$ and $i_!$ coincides with the Thom isomorphism \eqref{eq_Thom_isom}. 
\end{itemize}

In particular, if $X$ is a spin$^c$ manifold, we get the map
\begin{align*}
    \pi_{X!} \colon K^{\dim X }(X) \to K^0(pt) \simeq \Z,
\end{align*}
which sends $[E]$ to the index of the Dirac operator twisted by $E$ in the case $X$ is closed and even dimensional. 
A symplectic manifold has a canonical spin$^c$-structure, so we always use it to define $K$-theory pushforward. 

Finally, we explain the {\it Chern character homomorphism} which relates $K$-theory and ordinary cohomology. 
We have a map, 
\begin{align*}
    ch \colon K^i(X) \to \oplus_{n\in \Z}  H^{i+2n}(X; \Q), 
\end{align*}
which sends the class $[E]$ to its Chern character $\mathrm{Ch}(E)$. 
The chern character homomorhism commutes with pullback by a continuous map $f \colon X \to Y$.  However, it does not commute with pushforwards, and the difference is given by the Todd class. 
In particular we have
\begin{align}\label{eq_push_todd}
    \pi_{X!}([E]) = \int_X ch(E)td(X),
\end{align}
for an even dimensional spin$^c$-manifold $X$.

\section{The lattice index theorem}\label{sec_main}

In this section, we prove our main theorem, Theorem \ref{thm_lat_ind_inv.y}.  
The settings are as follows. 

Let $(B^n, \Lambda)$ be an $n$-dimensinal closed integral affine manifold, and let $X := T^*B /(2\pi \Lambda^*)$ be the cotangent torus bundle equipped with the canonical symplectic structure. 
Choose a prequantum line bundle $(L, \nabla^L)$ on $X$ whose Bohr-Sommerfeld points coincide with lattice points in $B$. 
This is always possible by Lemma \ref{lem_aff_Lag} (3). 
Then consider the associated Bohr-Sommerfeld deformation quantization maps, 
\begin{align*}
    \phi^k \colon C^\infty(X) \to \mathbb{B}(\mathcal{H}_k). 
\end{align*}
We extend these maps to matrix algebras naturally. 

Our main theorem, the lattice version of the Atiyah-Singer index theorem, is the following. 
Recall that an invertible self-adjoint element $f \in M_N(C^\infty(X))$ defines an element $[f] \in K^0(X)$, which is the class of the projection $(f(f^*f)^{-1/2} + 1)/2$. 

\begin{thm}[The lattice index theorem]\label{thm_lat_ind_inv.y}
Fix a positive integer $N$. 
Suppose we are given an invertible self-adjoint element $f \in M_N(C^\infty(X))$. 
Then there exists a positive integer $K$ such that, for all integer $k > K$, we have
\begin{align*}
  \mathrm{rank} \left(E_{>0}\left(\phi^k(f)\right)\right) 
  &= \pi_{X!}\left([L]^{\otimes k}\otimes [f]\right) \\
  &=(2\pi\sqrt{-1})^{-\dim B}\int_X ch(f)td(\omega)e^{\sqrt{-1}k\omega}. 
\end{align*}
Here $\pi_{X!} \colon K^0(X) \to K^0(pt)$ is the spin$^c$-pushforward map with respect to the canonical symplectic structure $\omega$ on $X$. 
\end{thm}

\subsection{The trace functional}
As a preperation to the proof of the main theorem, in this subsection we identify the canonical trace functional $\tau$ for the star product with the trace of operators in the Bohr-Sommerfeld deformation quantization, up to a constant. 

\begin{prop}\label{prop_trace}
Fix a function $f \in C^\infty(X)$. For any $N \in \N$ we have
\begin{align}\label{eq_prop_trace}
   \left| \mathrm{Trace}(\phi^k(f)) - \frac{k^n}{(2\pi)^n n!}\int_X f\omega^n\right| = O(k^{-N}), 
\end{align}
as $k \to \infty$. 
\end{prop}
\begin{proof}
Since the left hand side of the equation \eqref{eq_prop_trace} is linear in $f$, we may assume that $f$ is supported in a subset $\mu^{-1}(U) \subset X$ for some open set $U \subset B$, which has integral affine open embedding $U \hookrightarrow (\R/\Z)^n$, with an isomorphism of prequantum line bundle with the standard one in Example \ref{ex_preq}. 
Thus it is enough to consider the case $B= (\R/\Z)^n$ and $X = B \times (\R/2\pi \Z)^n$. 
In this case we have $B_k = \left(\frac{1}{k}\Z\right)^n/\Z^n$. 

By the definition of $\phi^k$ (Definition \ref{def_BSquantization}), we have
\begin{align*}
  \mathrm{Trace}(\phi^k(f)) =(2\pi)^{-n} \sum_{b \in B_k} \int_{X_b} f d\theta. 
\end{align*}
In partiular we see that both terms in the left hand side of \eqref{eq_prop_trace} are invariant if we take the fiberwise average of $f$, so we may assume that $f$ does not depend on the fiber variable. 
Also, it is enough to consider the case $n=1$. 
So it is enough to prove that, for any function $g \in C^\infty(\R/\Z)$ we have, for any $N$, 
\begin{align}\label{eq_proof_trace}
    \left|\frac{1}{k}\sum_{m=0}^{k-1}g\left(\frac{m}{k} \right) - \int_{\R/\Z} g(x)dx \right| = O(k^{-N}). 
\end{align}
This is elementary, seen as follows. 
Let us take the Fourier expansion $g = \sum_{l \in \Z}g_l e^{2\pi\sqrt{-1} lx}$. 
Then the left hand side of \eqref{eq_proof_trace} is bounded by $\sum_{|l| \ge k} |g_l|$. 
Since the Fourier coefficients of smooth functions are rapidly decreasing, we get the result. 
\end{proof}

By Proposition \ref{prop_trace} and Remark \ref{rem_trace}, we get the following. 
\begin{prop}\label{prop_trace_is_trace}
Let $f \in C^\infty(X)$. 
Then $\tau(f) \in \hbar^{-n}\C$. 
For each $k\in \N$ define $\tau_k(f) \in \C$ by setting $\hbar = (-\sqrt{-1})/k$ in $\tau(f)$. 
Then we have, for any $N \in \N$, 
\begin{align*}
    \left|\mathrm{Trace}(\phi^k(f)) - (2\pi \sqrt{-1})^{-n}\tau_k(f)
    \right| = O(k^{-N}). 
\end{align*}
\end{prop}

\subsection{The proof of Theorem \ref{thm_lat_ind_inv.y}}
In this subsection we prove Theorem \ref{thm_lat_ind_inv.y}. 
To simplify the notation, in this subsection we simply write $\star := \star_{MY}$. 

First we prove the following version of the theorem. 

\begin{thm}[The lattice index theorem, the projection formulation]\label{thm_lat_ind_proj.y}
Fix a positive integer $N$. 
Suppose we are given a projection $p_0 \in M_N(C^\infty(X))$ (with respect to the commutative product in $C^\infty(X)$). 
Let us denote the $K$-theory class of $p_0$ by $[p_0] \in K^0(X)$. 
Then there exists a positive integer $K$ such that, for all integer $k > K$, we have
\begin{align*}
  \mathrm{rank} \left(E_{>1/2}\left(\phi^k(p_0)\right) \right)
  &= \pi_{X!}\left([L]^{\otimes k}\otimes [p_0]\right) \\
  &= (2\pi\sqrt{-1})^{-\dim B}\int_X ch(p_0)td(\omega)e^{\sqrt{-1}k\omega}. 
\end{align*}
\end{thm}

The idea of the proof is as follows. 
Recall that the star product is realized as the composition of operators in Bohr-Sommerfeld deformation quantization (Proposition \ref{prop_star_is_comp}).  
The canonical trace functional is realized as the trace of operators (Proposition \ref{prop_trace_is_trace}). 
It is easy to see that the characteristic class of the standard Moyal-Weyl star product is simply $\omega \in H^2(X; \C[[\hbar]])$ (see the constructions of this class in \cite{NestTsygan1995a} or \cite{Fedosov1994}). 
Thus, we are in the settings of the algebraic index theorem. 

To prove Theorem \ref{thm_lat_ind_proj.y}, we extend a given projection $p_0$ in $M_N(C^\infty(X))$ (with respect to the commutative product on $C^\infty(X)$) to a projection in $M_N(C^\infty(X)[[\hbar]])$ (with respect to $\star$), and apply the algebraic index theorem to it. 

\begin{lem}\label{lem_extension.y}
Suppose we are given a projection $p_0 \in M_N(C^\infty(X))$ (with respect to the commutative product on $C^\infty(X)$). 
\begin{enumerate}
    \item There exists a unique element $p_{\hbar} \in M_N(C^\infty(X)[[\hbar]])$ such that
\begin{align*}
    p_{\hbar} \star p_{\hbar} =  p_{\hbar}, \
    p_{\hbar} = p_{\hbar}^*, \mbox{ and }
    p_{\hbar} = p_0 + O(\hbar). 
\end{align*}
Here, we introduce the $*$-algebra structure on $M_N(C^\infty(X))[[\hbar]]$ by setting
$\hbar = -\hbar^*$. 
\item
Let us write $p_\hbar = \sum_{i = 0}^\infty p_i \hbar^i$. 
For each positive integers $M $ and $k$, let us write
\begin{align}\label{eq_partial_sum.y}
    p^{M, k}:= \sum_{i = 0}^M p_i \left(\frac{-\sqrt{-1}}{k} \right)^i
    \in M_N(C^\infty(X)). 
\end{align}
Then, for each $M \in \N$, there exists a positive constant $C$ such that, for all $k \in \N$ we have 
\begin{align*}
    \left\|
    \phi^k(p^{M, k})^2 - \phi^k(p^{M, k})
    \right\| \le Ck^{-(M+1)}. 
\end{align*}
\end{enumerate}
 
\end{lem}

\begin{proof}
Set $u_0 := 2p_0 - 1$. 
This is a self-adjoint unitary element. 
For (1), it is enough to extend $u_0$ to an element $u_{\hbar}\in M_N(C^\infty(X))[[\hbar]]$ which is self-adjoint unitary with respect to $\star$, i.e., construct an element satisfying
\begin{align*}
    u_{\hbar} \star u_{\hbar} =  1, \
    u_{\hbar} = u_{\hbar}^*, \mbox{ and }
    u_{\hbar} = u_0 + O(\hbar). 
\end{align*}
We construct the coefficients $u_1, u_2, \cdots$ in the formal sum $u_\hbar = \sum_{i = 0}^\infty u_i \hbar^n$ inductively. 

Suppose that we have constructed $u_i$ for $1 \le i \le M-1$ such that, setting $u^{M-1} := \sum_{i = 0}^{M-1} u_i \hbar^i$, we have
\begin{align*}
    u^{M-1} \star u^{M-1} = 1 + O(\hbar^M) \mbox{ and }
    u^{M-1} = (u^{M-1})^*.  
\end{align*}
We construct $u_M$ such that 
\begin{align}
    (u^{M-1}+u_M\hbar^M) \star (u^{M-1}+u_M\hbar^M)  &= 1 + O(\hbar^{M+1}), \mbox{ and }\label{eq_lem_induction1.y}\\
    u_M &= (-1)^Mu_M^*\label{eq_lem_induction2.y}.  
\end{align}
Let us define $v \in M_N(C^\infty(X))$ by $u^{M-1} \star u^{M-1} = 1 + v\hbar^M +O(\hbar^{M+1})$. 
The associativity of $\star$ implies $u^{M-1} \star (u^{M-1} \star u^{M-1}) =( u^{M-1} \star u^{M-1})\star u^{M-1}$, and this implies
\begin{align}\label{eq_lem_commute.y}
    u_0 v = v u_0. 
\end{align}
The condition \eqref{eq_lem_induction1.y} is equivalent to
\begin{align*}
    u_M \star u_0 + u_0 \star u_M + v = O(\hbar), 
\end{align*}
so by \eqref{eq_lem_commute.y} it is enough to set $u_M:= -\frac{1}{2}u_0v$. 
Since $v\hbar^M = (v\hbar^M)^*$, we have $v = (-1)^M v^*$, so again using \eqref{eq_lem_commute.y} the condition \eqref{eq_lem_induction2.y} is also satisfied. 

By the proof above, the uniqueness is also clear. 

(2) follows from (1) and Proposition \ref{prop_star_is_comp}. 
\end{proof}

Now we can prove Theorem \ref{thm_lat_ind_proj.y}. 
\begin{proof}[Proof of Theorem \ref{thm_lat_ind_proj.y}]
The second equality follows from \eqref{eq_push_todd}. 
Let us set $2n = \dim X$. 
Suppose we are given a projection $p_0 \in M_N(C^\infty(X))$. 
Let us extend $p_0$ to a projection $p_\hbar = \sum_{i = 0}^\infty p_i \hbar^i \in M_N(C^\infty(X))[[\hbar]]$ as in Lemma \ref{lem_extension.y}. 
Then, applying the algebraic index theorem (Fact \ref{fact_alg_ind_thm.y}), we get
\begin{align}\label{eq_trace_formal.y}
    \tau (p_\hbar) = \int_X ch(p_0)td(\omega)e^{\omega/ \hbar}. 
\end{align}
Here we note that, in our case $c_1(\omega) = 0 \in H^2(X; \Q)$ because the $\omega$-compatible complex structure on $TX$ is the complexification of the real vector bundle $TB$. 
Also as noted before, we have $\theta = \omega$. 
Recall that, for all $f \in M_N(C^\infty(X))$ we have $\tau(f) \in \hbar^{-n}\C$ (Remark \ref{rem_trace}). 
Setting
\begin{align*}
    p_\hbar^n := \sum_{i = 0}^n p_i \hbar^i, 
\end{align*}
we see that
\begin{align}\label{eq_trace_partial.y}
    \tau (p_\hbar^n) = \int_X ch(p_0)td(\omega)e^{\omega/ \hbar} . 
\end{align}
For each positive integer $k$, we set
\begin{align*}
    p^{n, k}:= \sum_{i = 0}^n p_i \left(\frac{-\sqrt{-1}}{k} \right)^i \in M_N(C^\infty(X)), 
\end{align*}
as in \eqref{eq_partial_sum.y}. 
By \eqref{eq_trace_partial.y} and Proposition \ref{prop_trace_is_trace}, we see that
\begin{align}\label{eq_trace_operator.y}
    \mathrm{Trace}\left(\phi^{k}(p^{n, k})\right) 
    = (2\pi\sqrt{-1})^{-n}\int_X ch(p_0)td(\omega)e^{\sqrt{-1}k\omega}
    + O(k^{-1}). 
\end{align}
By Lemma \ref{lem_extension.y} (2), there exists a positive constant $C$ such that, for all $k$ we have
\begin{align}\label{eq_spec_gap.y}
  \mathrm{Spec}\left(\phi^k(p^{n, k})\right)
  \subset [-Ck^{-(n+1)}, Ck^{-(n+1)}] \cup [1 - Ck^{-(n+1)}, 1 + Ck^{-(n+1)}]. 
\end{align}
For each $k$, let us write 
\begin{align*}
    N(k) := \mathrm{rank} \left(E_{>1/2} \left( \phi^k(p^{n, k}) \right)\right) . 
\end{align*}
Then, for $k >(2C)^{1/(n+1)} $ we have
\begin{align*}
    \left| \mathrm{Trace}\left(\phi^{k}(p^{n, k})\right) - N(k) \right| \le Ck^{-{(n+1)}}\dim (\mathcal{H}_k \otimes \C^N). 
\end{align*}
Since $\dim (\mathcal{H}_k) = O(k^n)$, there exists a constant $D$ such that for all $k$, 
\begin{align}\label{eq_trace_spectrum.y}
     \left| \mathrm{Trace}\left(\phi^{k}(p^{n, k})\right) - N(k) \right| \le Dk^{-1}
\end{align}
Comparing equations \eqref{eq_trace_operator.y} and \eqref{eq_trace_spectrum.y}, and noting that the first term in the right hand side of \eqref{eq_trace_operator.y} is an integer, we see that, for $k$ large enough we have
\begin{align*}
    N(k) = (2\pi\sqrt{-1})^{-n}\int_X ch(p_0)td(\omega)e^{\sqrt{-1}k\omega}. 
\end{align*}
So the proof is reduced to showing the equation 
\begin{align}\label{eq_p0.y}
    N(k) = \mathrm{rank} \left(E_{>1/2}\left(\phi^k(p_0)\right)\right). 
\end{align}
Recall that, since the Bohr-Sommerfeld deformation quantization is a strict deformation quantization in the sense of Definition \ref{def_DQ} (\cite[Theorem 3.32]{Yamashita2020}), for any element $f\in C^\infty(X)$, we have
\begin{align*}
    \lim_{k \to \infty}\|\phi^k(f)\| = \|f\|_{C^0}. 
\end{align*}
So we have
\begin{align*}
    \|\phi^k(p^{n, k} - p_0)\|
    =O(k^{-1}). 
\end{align*}
Combining this and \eqref{eq_spec_gap.y}, we see that, for $k$ large enough we have
\begin{align*}
     \mathrm{rank} \left(E_{>1/2}\left(\phi^k(p_0)\right)\right)
     =
       \mathrm{rank} \left(E_{>1/2} \left( \phi^k(p^{n, k}) \right)\right) , 
\end{align*}
so \eqref{eq_p0.y} follows. 
\end{proof}

Next we use Theorem \ref{thm_lat_ind_proj.y} to prove Theorem \ref{thm_lat_ind_inv.y}. 

\begin{lem}\label{lem_lower_bd.y}
Let $f \in M_N(C^\infty(X))$ be an invertible element. 
Then for any $\epsilon >0 $ there exists an integer $k_0$ such that for all $k > k_0$ we have
\begin{align*}
    \left| \phi^k(f)\right| > \frac{1}{\|f^{-1}\|_{C^0}} - \epsilon. 
\end{align*}
\end{lem}
\begin{proof}
This follows from
\begin{align*}
    \phi^k(f) \phi^k(f^{-1}) &= 1 + O(k^{-1}) \quad \mbox{ and }\\
    \lim_{k \to \infty} \|\phi^k(f^{-1})\| &= \|f^{-1}\|_{C^0}. 
\end{align*}
\end{proof}

\begin{proof}[Proof of Theorem \ref{thm_lat_ind_inv.y}]
If the element $f$ is self-adjoint unitary, the statement follows directly from Theorem \ref{thm_lat_ind_proj.y}, applied to the self-adjoint projection $p_0 := (f+1)/2 $. 
In the general case, we can reduce to the case of a self-adjoint unitary as follows. 
Given an invertible self-adjoint element $f \in M_N(C^\infty(X))$, set $u := f(f^*f)^{-1/2}$. 
Then it is enough to show that, 
\begin{align}\label{eq_dim_coincide.y}
    \mathrm{rank} \left(E_{>0}\left(\phi^k(f)\right)\right) = \mathrm{rank} \left(E_{>0}\left(\phi^k(u)\right)\right)
    \mbox{ if }k >>0. 
\end{align}
By Lemma \ref{lem_lower_bd.y}, for $k$ large enough, all of the operators $\phi^k(f)$, $ \phi^k((f^*f)^{-1/4})$ and $\phi^k(u)$ are self-adjoint and invertible. 
Moreover, for $k$ large enough $\phi^k((f^*f)^{-1/4})$ is a positive operator. 
Indeed we can apply Lemma \ref{lem_lower_bd.y} again to $(f^*f)^{-1/8}$ and use the estimate
\begin{align*}
    \left\| \phi^k((f^*f)^{-1/8})^2  - \phi^k((f^*f)^{-1/4}) \right\| = O(k^{-1}).
\end{align*}
For such $k$, we have
\begin{align*}
    \mathrm{rank} \left(E_{>0}\left(\phi^k(f)\right) \right)= \mathrm{rank} \left(E_{>0} \left(
    \phi^k((f^*f)^{-1/4})\phi^k(f) \phi^k((f^*f)^{-1/4})\right)\right). 
\end{align*}
Moreover we have
\begin{align*}
    \left\| \phi^k((f^*f)^{-1/4})\phi^k(f) \phi^k((f^*f)^{-1/4}) - \phi^k(u) \right\| = O(k^{-1}).
\end{align*}
Since we have $|\phi^k(u)|> 1/2 $ for $k$ large enough, we get \eqref{eq_dim_coincide.y} and the proof is complete. 
\end{proof}

\section{An application : The index problem of the Wilson-Dirac operator}\label{sec_application}
In this section, we apply the lattice index theorem (Theorem \ref{thm_lat_ind_inv.y}) to the index problem of the Wilson-Dirac operator as explained in the introduction. 
We are interested in the index of twisted spin Dirac operators on an even-dimensional torus. 
We want to recover the continuum index from some operators on lattice. 
In order for this, we use operators called {\it ``Wilson-Dirac operators''} (Definition \ref{def_lat_op.y}).   
The main theorem is Theorem \ref{thm_application}, which recovers the result by Adams \cite{Adams2001}. 
It relates the index of the Dirac operator on the continuum limit with the dimension of positive eigenspaces of Wilson-Dirac operators. 
The existing proof for this fact are done by analysis of index density called Fujikawa's method. 
The argument here can be regarded as a new topological proof for it.  

Other approaches to the problems treated in this section will appear in \cite{FFMOYY2020} and \cite{Kubota2020}. 

The setting is as follows. 
\begin{itemize}
    \item Let us fix an even positive integer $n \in 2\Z_{>0}$ and let $B:=  (\R / \Z)^n$. 
    We consider the standard flat metric and translation-invariant spin structure on $B$. 
    \item Let $\C l_n$ denote the complex Clifford algebra with generators $\{c_i\}_{i = 1}^n$ satisfying $c_ic_j + c_j c_i = - 2\delta_{ij}$, $c_i = -c_i^*$. 
    Let $S$ denote the spinor space, the irreducible representation space of $\C l_n$. 
    Let us denote by $\Gamma \in \mathrm{End}(S)$ the $\Z_2$-grading operator on $S$. 
    We have $\Gamma c_i + c_i \Gamma = 0$. 
    \item
    The spinor bundle on $B$ is identified with the product bundle $\underline{S} =B \times S$ equipped with the Clifford action $c(dx^i) = c_i$. 
    \item Assume we are given a smooth hermitian vector bundle with unitary connection $(E, \nabla^E)$ over $B$. 
    \item Below we consider operators on the vector bundle $\underline{S} \otimes E$. 
    The $\Z_2$-grading operator on this bundle is denoted by $\gamma := \Gamma \otimes \mathrm{id}_E$.  
    \item We consider the standard integral affine structure on $B$. 
    We use the prequantum line bundle on $T^*B/(2\pi \Lambda^*) = (\R/\Z)^n \times (\R /(2 \pi\Z))^n $ defined in Example \ref{ex_preq}. 
    The set of level $k$-lattice points is given by $B_k =\left(\frac{1}{k}\Z^n\right)/\Z^n$. 
    In this case we have the canonical identification $\mathcal{H}_k \simeq l^2(B_k)$ (see the proof of Lemma \ref{lem_aff_Lag} (3)). 
\end{itemize}

On the continuum side, we have the twisted spin Dirac operator defiend as follows. 
\begin{defn}
The spin Dirac operator on $B$ twisted by $(E, \nabla^E)$, denoted by  $D^{\mathrm{conti}} \colon L^2(B; \underline{S} \otimes E) \to L^2(B; \underline{S} \otimes E)$, is defined by
\begin{align*}
    D^{\mathrm{conti}} := \sum_{i = 1}^n c_i \nabla_{\del_i}^{\underline{S} \otimes E}. 
\end{align*}
Here $\nabla^{\underline{S} \otimes E}$ is the tensor product connection of $(E, \nabla^E)$ and the trivial connection on $\underline{S}$, and we write $\del_i := \frac{\del}{\del x_i}$. 
This is an odd (i.e., $\gamma D^{\mathrm{conti}} + D^{\mathrm{conti}} \gamma = 0 $) and self-adjoint elliptic operator. 
\end{defn}

On the lattice side, we define the following operators. 

\begin{defn}\label{def_lat_op.y}
For each positive integer $k$ and $i = 1, \cdots, n$, we define the following operators $U_{k, i}$, $\nabla_{k, i}$, $D^{\mathrm{lat}}_k$ and $W_k$ on $l^2(B_k; (\underline{S} \otimes E)|_{B_k})$. 
\begin{itemize}
    \item For each $i= 1, \cdots, n$ and $x \in B_k$, let us denote by $T_{k, i, x} \colon (\underline{S} \otimes E)|_{x} \to (\underline{S} \otimes E)|_{x + e_i/k}$ the parallel transport map with respect to $\nabla^{\underline{S}\otimes E}$ along the path $x + te_i/k$, $t \in [0, 1]$. 
    Here we denoted the $i$-th unit vector on $\R^n$ by $e_i$. 
    The forward shift operator $U_{k, i}$ is defined by $U_{k, i} := \oplus_{x \in B_k} T_{k, i, x}$. 
    \item For each $i = 1, \cdots, n$, the forward-differential $\nabla_{k, i}$ is defined by
    \begin{align*}
        \nabla_{k, i}:= k(U_{k, i}^* - 1). 
    \end{align*}
    \item The level-$k$ lattice Dirac operator $D^{\mathrm{lat}}_k$ is defined by
    \begin{align*}
        D^{\mathrm{lat}}_k := \sum_{i = 1}^n c_i \frac{\nabla_{k, i} - \nabla_{k, i}^*}{2}. 
    \end{align*}
    \item 
    The Wilson term $W_k$ is defined by
    \begin{align*}
        W_k := \sum_{i = 1}^n \frac{\nabla_{k, i} + \nabla_{k, i}^*}{2}. 
    \end{align*}
\end{itemize}
Fixing a positive constant $r > 0$, the operator $D_k^{\mathrm{lat}} + r\gamma W_k$ is called the {\it Wilson-Dirac operator}. 
\end{defn}

\begin{rem}
When $(E, \nabla^E)$ is trivial, we have
\begin{align*}
    (\nabla_{k, i} f)(x) &= \frac{f(x + e_i/k) - f(x)}{1/k},  \\
    (D^{\mathrm{lat}}_k f )(x)&= \sum_{i = 1}^n c_i \frac{f(x + e_i/k) - f(x- e_i/k)}{2/k}, \\
     (W_k f)(x) &= \sum_{i = 1}^n \frac{f(x + e_i/k) - 2f(x)+f(x - e_i/k) }{2/k}.
\end{align*}
\end{rem}

In order to state the main result, we define an integer $I_n(m)$ for $m \in \R \setminus \{0, 2, 4, \cdots, 2n\}$ as follows. 

\begin{defn}\label{def_coeff}
For $m \in \R \setminus \{0, 2, 4, \cdots, 2n\}$, we define an integer $I_n(m)$ as follows. 
For $2l < m < 2l+2$ with $l = 0, 1, \cdots, n-1$, we set
\begin{align*}
    I_n(m) := \sum_{i = 0}^l (-1)^i \binom{n}{i}. 
\end{align*}
For $m \notin [0, 2n]$ we set
\begin{align*}
    I_n(m) := 0. 
\end{align*}
\end{defn}

\begin{ex}
When $n = 4$, we have
$I_4(m) = 1$ for $0 < m < 2$,
$I_4(m) = -3$ for $2 < m < 4$,
$I_4(m) = 3$ for $4 < m < 6$,
$I_4(m) = -1$ for $6 < m < 8$, and
$I_4(m) = 0$ for $m \notin [0, 8]$. 

In general $I_n(m) = 1$ for $0 < m < 2$, $I_n(m) = 1-n$ for $2 < m < 4$, and so on. 
\end{ex}

The following proposition is crutial in the proof of Theorem \ref{thm_application}. 
\begin{prop}\label{prop_DW_class}
For $m \in \R \setminus \{0, 2, 4, \cdots, 2n\}$ and $r>0$, let us define $f_{DW}(m, r) \in C((\R /( 2\pi \Z))^n) \otimes \mathrm{End}(S)$ by
\begin{align}\label{eq_fDW_def}
    f_{DW}(m, r) := \sum_{i = 1}^n \left\{-\sqrt{-1}c_i\sin \theta_i + r\Gamma \left(\cos \theta_i - 1\right) \right\} + rm\Gamma.
\end{align}
Then this element is invertible and self-adjoint. 
Moreover we have the following equality in $K^0((\R / (2\pi \Z))^n)$, 
\begin{align}\label{eq_Kclass}
    I_n(m)\cdot i_{pt!} \left([1]\right) = [f_{DW}(m, r)] - [-\Gamma] . 
\end{align}
Here we denoted the inclusion of a point by $i_{pt} \colon \{pt\} \hookrightarrow (\R / (2\pi \Z))^n$ and $[1]\in K^0(pt)$ is the generator. 
The two elements in the right hand side of \eqref{eq_Kclass} are the classes of the invertible self-adjoint elements $f_{DW}(m, r)$ and $-\Gamma$ in $C((\R /( 2\pi \Z))^n) \otimes \mathrm{End}(S)$. 
\end{prop}

\begin{proof}
In this proof we denote $Z := (\R / (2\pi \Z))^n$. 
Using the relations $c_ic_j + c_jc_i= -2\delta_{ij}$ and $\Gamma c_i + c_i \Gamma = 0$, we have
\begin{align}\label{eq_square_lem}
    \left(f_{DW}(m, r)  \right)^2
    = \left\{\sum_{i = 1}^n \sin^2 \theta_i + r^2\left(\sum_{i = 1}^n (\cos \theta_i - 1) + m
    \right)^2 \right\} 
\end{align}
Since we have assumed that $m \neq 0, 2, \cdots, 2n$, we see that $f_{DW}(m, r)$ is invertible. 

Now we prove \eqref{eq_Kclass}. 
For a real number $s$, let us denote
\begin{align*}
    Y_s := \{(\theta_1, \cdots, \theta_n) \in Z\ | \ |\sin \theta_i| \le s \mbox{ for all }i\}. 
\end{align*}
In particular $Y_0$ is a set consisting of $2^n$-points. 
First we construct a homotopy between the element in the right hand side of \eqref{eq_Kclass} and an element supported in $\mathring{Y}_{0.2}$. 
Fix any continuous function $\kappa \colon Z\to [0,1]$ such that $\kappa = 0$ on $Z \setminus Y_{0.2}$ and $\kappa = 1$ on $Y_{0.1}$. 
We claim that, in $K^0(Z)$ we have
\begin{align}\label{eq_htpy}
    [f_{DW}(m, r)] = [\kappa f_{DW}(m, r) -(1 - \kappa)\Gamma]. 
\end{align}
Indeed, the linear homotopy does the job; for any $0 \le t \le 1$, by a computation similar to \eqref{eq_square_lem},we easily see that $(1-t(1 -\kappa))f_{DW}(m, r) - t(1-\kappa)\Gamma$ is invertible and self-adjoint (here it is crucial that $\kappa = 1$ on a neighborhood of $Y_0$). 

Since $\kappa f_{DW}(m, r) -(1 - \kappa)\Gamma = -\Gamma$ on $Z \setminus Y_{0.2}$, we see that the class $[\kappa f_{DW}(m, r) -(1 - \kappa)\Gamma] - [-\Gamma]$ is supported in $\mathring{Y}_{0.2}$, so we are left to evaluate contributions from each component of $\mathring{Y}_{0.2}$.  

\begin{lem}\label{lem_contribution}
Fix a point $p = (\theta_1(p), \cdots, \theta_n(p)) \in Y_0 = \{0, \pi\}^n$ and denote the connected component of $\mathring{Y}_{0.2}$ containing $p$ by $U_p$. 
Define
\begin{align*}
    \epsilon(p) := \sharp \{i \in \{1, 2, \cdots, n\} \ | \ \theta_i(p) = \pi \}. 
\end{align*}
Then we have the following equalities in $K^0(U_p)$. 
\begin{enumerate}
    \item If $\sum_{i = 1}^n( \cos \theta_i(p)-1) + m < 0 $, we have
    \begin{align*}
        \left.\left([\kappa f_{DW}(m, r) -(1 - \kappa)\Gamma] - [-\Gamma]\right)\right|_{U_p} = 0. 
    \end{align*}
    \item If $\sum_{i = 1}^n( \cos \theta_i(p)-1) + m > 0 $, we have
    \begin{align*}
        \left.\left([\kappa f_{DW}(m, r) -(1 - \kappa)\Gamma] - [-\Gamma]\right)\right|_{U_p} 
        = (-1)^{\epsilon(p)} \cdot i_{p!}([1]). 
    \end{align*}
\end{enumerate}
\end{lem}
\begin{proof}
Restricted to $U_p$, the element $\kappa f_{DW}(m, r) -(1 - \kappa)\Gamma$ is homotopic (in the space of invertible and self-adjoint elements which coincides with $-\Gamma$ outside a compact set) to the element
\begin{align}\label{eq_elem_neighbor}
   \kappa \left( \sum_{i = 1}^n -(-1)^{\epsilon_i(p)}\sqrt{-1}c_i (\theta_i - \theta_i(p)) + r\Gamma\left(\sum_{i = 1}^n( \cos \theta_i(p)-1) + m\right)\right) - (1-\kappa)\Gamma. 
\end{align}
Here $\epsilon_i(p) := 0$ if $\theta_i(p) = 0$ and $\epsilon_i(p) := 1$ if $\theta_i(p) = \pi$. 

If $\sum_{i = 1}^n( \cos \theta_i(p)-1) + m < 0 $, we can connect the element \eqref{eq_elem_neighbor} to $-\Gamma$ by the linear homotopy, so we get (1). 

For (2), recall that the element $i_{pt!}([1])\in K^0(\R^n)$ is represented by elements of $C_c(\R^n)^+ \otimes \mathrm{End}(S)$ as\footnote{$C_c(\R^n)^+$ denotes the unitization of $C_c(\R^n)$. } (see Remark \ref{rem_bott} below)
\begin{align}\label{eq_bottelem}
    \left[\kappa \left( \sum_{i = 1}^n-\sqrt{-1} c_i x_i + \Gamma \right) - (1 - \kappa)\Gamma \right] - [-\Gamma]. 
\end{align}
Here $(x_1, \cdots, x_n)$ is the coordinate on $\R^n$ and $\kappa$ is the cutoff function which is $1$ at the origin and $0$ outside a compact set. 
If we flip the sign of some of $x_i$'s, the sign of the resulting element changes accordingly. 

If $\sum_{i = 1}^n( \cos \theta_i(p)-1) + m > 0 $, the element \eqref{eq_elem_neighbor} is homotopic to the first term of \eqref{eq_bottelem} with $\epsilon(p)$-times of change of signs in $x_i$'s, so we get (2). 
This completes the proof of Lemma \ref{lem_contribution}. 
\end{proof}
Let 
\begin{align*}
    Y'_0 := \{p \in Y_0 \ | \ \sum_{i = 1}^n( \cos \theta_i(p)-1) + m > 0 \}. 
\end{align*}
By \eqref{eq_htpy} and Lemma \ref{lem_contribution}, we have
\begin{align*}
     [f_{DW}(m, r)] - [-\Gamma]&= [\kappa f_{DW}(m, r) -(1 - \kappa)\Gamma] - [-\Gamma] \\
     &=\sum_{p \in Y_0} i_{U_p!}\left( \left.\left([\kappa f_{DW}(m, r) -(1 - \kappa)\Gamma] - [-\Gamma]\right)\right|_{U_p}\right) \\
     &=\left( \sum_{p \in Y'_0}(-1)^{\epsilon(p)}\right) i_{pt!}([1]). 
\end{align*}
Here we denoted by $i_{U_p} \colon U_p \hookrightarrow Z$ the inclusion and by $i_{U_p!} \colon K^0(U_p) \to K^0(Z)$ the associated pushforward map. 
The second equality follows from the fact that the class $[\kappa f_{DW}(m, r) -(1 - \kappa)\Gamma] - [-\Gamma]$ is supported in $\mathring{Y}_{0.2}$. 
It is easy to see that 
\begin{align*}
    I_n(m) = \sum_{p \in Y'_0}(-1)^{\epsilon(p)}. 
\end{align*}
So we get \eqref{eq_Kclass}. 
This completes the proof of Proposition \ref{prop_DW_class}. 
\end{proof}

\begin{rem}\label{rem_bott}
Here we explain that the element \eqref{eq_bottelem} gives the generator $\beta := i_{pt!}([1]) \in K^0(\R^n)$. 
If we use the picture of the $K$-group representing a class by a $\Z_2$-graded vector bundle with an odd self-adjoint endomorphism which is invertible outside a compact set, 
$\beta$ is represented by the class $[S, \sigma]$, where $\sigma$ denotes the Clifford multiplication $\sigma(x_i) = -\sqrt{-1}c_i x_i$ (see \cite[Chapter 1, Remark 9.28]{LawsonMichelsohn1989}). 
To see that this class is the same as \eqref{eq_bottelem}, 
we first renormalize this class by setting
$\tilde{\sigma} := \chi (\|\sigma\|)\|\sigma\|^{-1} \sigma$, where $\chi$ is a continuous function $\chi \colon [0, \infty) \to [0, 1]$ so that $\chi = 1$ outside a compact set. 
Then we have $\beta = [S, \tilde{\sigma}]$. 
Recall that we have been representing an element of $K^0$-group using self-adjoint invertible endomorphisms. 
The odd self-adjoint Fredholm picture (with $\|F\| =1$ and $F^2 - 1$ is compact) and the ungraded self-adjoint unitary picture of $K^0$-groups are related by the map (see \cite[Proposition 4.3]{AS1969})
\begin{align*}
    [F] \mapsto [\Upsilon \exp(\Upsilon F \pi)] - [-\Upsilon] = [\Upsilon \cos (F\pi) + \sin (F\pi)] - [-\Upsilon]. 
\end{align*}
Here we denoted the $\Z_2$-grading operator by $\Upsilon$. 
In our case, applying the above correspondence to the element $[S, \tilde{\sigma}]$ we get an element 
homotopic to \eqref{eq_bottelem}. 
\end{rem}

\begin{thm}\label{thm_application}
Fix constants $r, m$ so that $r > 0$ and $m \in \R \setminus \{0, 2, 4, \cdots, 2n\}$. 
Then for $k$ large enough 
we have
\begin{align}\label{eq_DW_thm}
    I_n(m)\mathrm{Ind}(D^{\mathrm{conti}}) 
     = \mathrm{rank} \left(E_{>0}\left(D^{\mathrm{lat}}_{ k} + r\gamma(W_k + mk)\right)\right) - \frac{1}{2} \dim l^2(B_k; (\underline{S} \otimes E)|_{B_k}) . 
\end{align}
Here $I_n(m) \in \Z$ is defined in Definition \ref{def_coeff}. 
\end{thm}

\begin{proof}
We fix an integer $N$ and an embedding of complex vector bundle $E \hookrightarrow \underline{\C^N}= B \times \C^N$ preserving the metric. 
    We denote by $p \in M_N(C^\infty(B))$ the projection corresponding to $E$. 
    We have $[E] = [p]$ in $K^0(B)$. 
    
Let $X := T^*B / (2\pi \Lambda^*) = B \times (\R / (2\pi \Z))^n$ be the cotangent torus bundle over $B$.
We have the Bohr-Sommerfeld deformation quantization maps extended canonically to the matrix algebra\footnote{
Precisely, we need to specify an open covering $\mathcal{U}$ of $B$ (see subsection \ref{subsec_preliminaries_BS}). 
Since all the operators appearing in this proof only contain shifts up to $\pm 1/k$ on $B_k$, any covering in which all neighboring pairs of points in $B_k$ are close (i.e., contained in some common element in $\mathcal{U}$) and satisfies the condition (U), produces the same quantization map $\phi^k$.   
We are only interested in behaviors of operators as $k \to \infty$, so the choice of $\mathcal{U}$ does not matter. 
}, 
\begin{align*}
    \phi^k \colon C^\infty(X)\otimes \mathrm{End}(S \otimes\C^N) \to
    \B(\mathcal{H}_k)\otimes \mathrm{End}(S \otimes\C^N)
    = \mathrm{End}\left(l^2(B_k;(\underline{S} \otimes \underline{\C^N})|_{B_k})\right) .
\end{align*}
We identify $\mathrm{End}\left(l^2(B_k;(\underline{S} \otimes E)|_{B_k} )\right) $ with a subalgebra of $\mathrm{End}\left(l^2(B_k;(\underline{S} \otimes \underline{\C^N})|_{B_k} )\right) $ canonically. 
Abusing the notation, we also write $p := \mu^*p \in C(X) \otimes \mathrm{End}(S \otimes\C^N)$.

By the definition of $\phi^k$, the operator $\phi^k(e^{\sqrt{-1}\theta_i})$ is the forward-shift operator in $i$-th direction on the lattice $B_k$ (see Example \ref{ex_concentration}). 
Since $p$ and $\nabla^E$ are smooth, there exists a constant $A>0$ independent of $k$ and $i$ such that we have
\begin{align*}
     \left\| U_{k, i}- \phi^k\left(e^{\sqrt{-1}\theta_i} \otimes p\right)\right\| < Ak^{-1}. 
\end{align*}
Recall that (Definition \ref{def_lat_op.y} and \eqref{eq_fDW_def}), 
\begin{align*}
    D^{\mathrm{lat}}_{k} + r\gamma (W_k + mk)
    &= k \sum_{i}^n \left( c_i \frac{U_{k, i}^* - U_{k, i}}{2} + r\gamma \frac{U_{k, i}^* - U_{k, i} - 2}{2}\right) + mk\gamma, \\
    f_{DW}(m, r) &=  \sum_{i = 1}^n \left\{-\sqrt{-1}c_i\sin \theta_i + r\Gamma \left(\cos \theta_i - 1\right) \right\} + rm\Gamma.
\end{align*}
We get, setting $A' := 2nA$,
\begin{align}\label{eq_DW_symb.y}
    \left\|\left(D^{\mathrm{lat}}_{k} + r\gamma (W_k + mk)\right)
    - k\phi^k\left(f_{DW}(m, r)
    \otimes p\right) 
    \right\| < A'. 
\end{align}
Here we pullback the element $f_{DW}(m, r) \in C^\infty((\R / (2\pi \Z))^n) \otimes \mathrm{End}(S)$ defined in \eqref{eq_fDW_def} to $X$, trivially in the $B$-direction, and still denote it by $f_{DW}(m, r) \in C^\infty(X) \otimes \mathrm{End}(S)$. 
Since by Proposition \ref{prop_DW_class} the element $f_{DW}(m, r) \otimes p - \Gamma \otimes (1-p)$ is invertible, by Lemma \ref{lem_lower_bd.y} there exists a positive constant $C > 0$ such that, for $k$ large enough we have
\begin{align*}
    \left|\phi^k\left(f_{DW}(m, r) \otimes p - \Gamma \otimes (1-p)\right)\right| > C. 
\end{align*}
From this and \eqref{eq_DW_symb.y}, we see that, for $k$ large enough, 
\begin{align}\label{eq_op_symb2}
   &\mathrm{rank}\left(E_{>0} \left(\phi^k\left(f_{DW}(m, r) \otimes p - \Gamma \otimes (1-p)\right)\right)\right) \\
   &=\mathrm{rank}\left(E_{>0}\left(\frac{1}{k}\left\{D^{\mathrm{lat}}_{ k} + r\gamma(W_k + mk)\right\} +\phi^k\left(- \Gamma \otimes (1-p)\right)\right)\right) \notag \\
   &= \mathrm{rank}\left(E_{>0}\left(D^{\mathrm{lat}}_{ k} + r\gamma(W_k + mk)\right)\right) + 
   \mathrm{rank}\left(E_{>0} \left(\phi^k(-\Gamma\otimes (1-p))\right)\right), \notag
\end{align}
so that
\begin{align}\label{eq_op_symb}
    &\mathrm{rank}\left(E_{>0} \left(\phi^k\left(f_{DW}(m, r) \otimes p - \Gamma \otimes (1-p)\right)\right)\right) - 
\mathrm{rank}\left(E_{>0} \left(\phi^k(-\Gamma\otimes \mathrm{id}_{\C^N})\right)\right) \\
 &=\mathrm{rank}\left(E_{>0} \left(\phi^k\left(f_{DW}(m, r) \otimes p - \Gamma \otimes (1-p)\right)\right)\right) - 
\mathrm{rank}\left(E_{>0} \left(\phi^k(-\Gamma\otimes p)\right)\right) \notag \\
    &= \mathrm{rank}\left(E_{>0}\left(D^{\mathrm{lat}}_{ k} + r\gamma(W_k + mk)\right)\right)
    -\frac{1}{2} \dim l^2(B_k; (\underline{S} \otimes E)|_{B_k}). \notag
\end{align}
 
Now we claim that, in $K^0(X)$ we have
\begin{align}\label{eq_Lag}
    \left([f_{DW}(m, r)] - [-\Gamma]\right)  \otimes \left([L] - 1 \right) = 0. 
\end{align}
Indeed, the element in the left hand side of \eqref{eq_Lag} is equal to the pullback of the element in the right hand side of \eqref{eq_Kclass} via the natural map $X \to (\R/ (2\pi\Z))^n$. 
This means that, by Proposition \ref{prop_DW_class}, 
\begin{align}\label{eq_push}
    \left([f_{DW}(m, r)] - [-\Gamma]\right) 
    = I_n(m) \cdot i_{B!}([1]) \in K^0(X), 
\end{align}
where $i_B \colon B \hookrightarrow X$ is the inclusion to the zero section. 
Since the zero section of $X$ is a Lagrangian submanifold of $X$, the restriction of the prequantum line bundle $L$ to the zero section is trivial (note that we do not have torsion in $K^0(X)$). 
Thus by excision we have 
\begin{align*}
\mathrm{im}(i_{B!}) \subset \ker\left(\left([L] - 1\right) \otimes \cdot\right) \mbox{ in }K^0(X). 
\end{align*}
So we get \eqref{eq_Lag}. 

By Theorem \ref{thm_lat_ind_inv.y}, for $k$ large enough, 
\begin{align}\label{eq_DW_index.y}
     & 
  \mathrm{rank}\left(E_{>0} \left(\phi^k\left(f_{DW}(m, r) \otimes p - \Gamma \otimes (1-p)\right)\right)\right) - 
\mathrm{rank}\left(E_{>0} \left(\phi^k(-\Gamma\otimes \mathrm{id}_{\C^N})\right)\right)
   \\ \notag
  &= \pi_{X!}\left([L]^{\otimes k}\otimes \left( \left[f_{DW}(m, r) \otimes p - \Gamma \otimes (1-p)\right] - [-\Gamma \otimes \mathrm{id}_{\C^N}]\right)\right) \\ \notag
  &= \pi_{X!}\left([L]^{\otimes k}\otimes \left( \left[f_{DW}(m, r)\right] - [-\Gamma]\right) \otimes [p]\right) \\ \notag
  &=\pi_{X!}\left( \left( \left[f_{DW}(m, r)\right] - [-\Gamma]\right) \otimes [p]\right) \quad \mbox{by \eqref{eq_Lag}} \\ \notag
  &= I_n(m)\cdot \pi_{X!}(i_{B!}[1] \otimes [p])
  \quad \mbox{by \eqref{eq_push}} \\ \notag
  &= I_n(m) \cdot \pi_{B!}([E]). 
\end{align}
Here we denoted the spin$^c$-pushforward map for $B$ by $\pi_{B!} \colon K^0(B) \to K^0(pt)$. 

On the other hand, by the Atiyah-Singer index theorem we have
\begin{align}\label{eq_AS_application}
    \mathrm{Ind}(D^{\mathrm{conti}}) = \pi_{B!}([E]). 
\end{align}
Thus, conbining \eqref{eq_op_symb}, \eqref{eq_DW_index.y} and \eqref{eq_AS_application}, we get the result. 
\end{proof}

Here we prove a similar but different version of the result, which is used, for example, in \cite{FKMMNOY2020}.  
\begin{cor}
In the above settings, there exists a constant $M_0 > 0$ such that, for all $M > M_0$, for $k$ large enough 
we have
\begin{align}
    \mathrm{Ind}(D^{\mathrm{conti}}) 
     = \mathrm{rank} \left(E_{>0}\left(D^{\mathrm{lat}}_{ k} + \gamma(W_k + M)\right)\right) - \frac{1}{2} \dim l^2(B_k; (\underline{S} \otimes E)|_{B_k}) . 
\end{align}
\end{cor}
\begin{proof}
We will use Theorem \ref{thm_application} for the case $r = 1$ and $m = 0.5$. 
We have the following. 

\begin{lem}\label{lem_posrank}
There exists a constant $M_0> 0$ such that for each $M > M_0$, for $k$ large enough we have
\begin{align}\label{eq_lem_posrank}
    &\mathrm{rank} \left(E_{>0}\left(\phi^k\left(f_{DW}(0.5, 1) \otimes p - \Gamma \otimes (1-p)\right)\right)\right) \\
    &=
    \mathrm{rank} \left(E_{>0}\left(D^{\mathrm{lat}}_{ k} + \gamma(W_k + M)\right)\right) + \mathrm{rank} \left(E_{>0}\left(-\Gamma\otimes (1-p)\right)\right).  \notag
\end{align}
Here, all the operators appearing in the equation are on $l^2(B_k;(\underline{S} \otimes \underline{\C^N})|_{B_k})$
\end{lem}

\begin{proof}
Let $M> 0$ be an arbitrary positive number. 
Since $D^{\mathrm{lat}}_{ k} + \gamma(W_k + M)$ and $p$ commute, the right hand side of \eqref{eq_lem_posrank} is equal to
\begin{align*}
    \mathrm{rank} \left(E_{>0}\left(\left(D^{\mathrm{lat}}_{ k} + \gamma(W_k + M)\right)-\Gamma\otimes (1-p)\right)\right). 
\end{align*}
We have (note that $\phi^k(\Gamma \otimes p) = \Gamma \otimes p = \gamma$)
\begin{align}\label{eq_lem_posrank1}
    &\left\|k\phi^k\left(f_{DW}(0.5, 1) \otimes p - \Gamma \otimes (1-p)\right)
    - \left(\left(D^{\mathrm{lat}}_{ k} + \gamma(W_k + M)\right)-\Gamma \otimes (1-p)\right) \right\|\\
    &=\left\|k\phi^k\left(f_{DW}(0.5, 1) \otimes p\right) - \left(D^{\mathrm{lat}}_{ k} + \gamma (W_k+0.5k) \right)+ (0.5k-M)\gamma \right\| \notag\\
    &\le A' + |0.5k - M| \notag
\end{align}
The last inequality follows from \eqref{eq_DW_symb.y}. 
On the other hand, by \eqref{eq_square_lem}, we have
\begin{align*}
    |f_{DW}(0.5, 1)| \ge 0.5. 
\end{align*}
Since $f_{DW}(0.5, 1)$ does not depend on the base variable (i.e., translation-invariant on $B$), by definition of $\phi^k$ we see easily that (for example regard the operator $\phi^k(f_{DW}(0.5, 1))$ as a convolusion operator on the group $(\Z /k\Z)^n$),
\begin{align*}
    |\phi^k(f_{DW}(0.5, 1))| \ge 0.5. 
\end{align*}
From this and using the fact that $\phi^k(f_{DW}(0.5, 1))$ only contains one-shift on the lattice $B_k$ and the smoothness of $p$ and $\nabla^E$, we easily see that there exists a constant $D>0$ such that for all $k$, 
\begin{align}\label{eq_lem_posrank2}
   \left| k\phi^k\left(f_{DW}(0.5, 1) \otimes p - \Gamma\otimes (1-p)\right)\right|
   > 0.5 k - D. 
\end{align}

Now put $M_0 := A' + D$. Then by \eqref{eq_lem_posrank1} and \eqref{eq_lem_posrank2} we see that it satisfies the condidion. 
\end{proof}
Set the constant $M_0>0$ so that it satisfies the condition in Lemma \ref{lem_posrank}. 
Take any constant $M > M_0$. 
From Lemma \ref{lem_posrank} and \eqref{eq_op_symb2}, we see that for $k$ large enough, 
\begin{align*}
     \mathrm{rank} \left(E_{>0}\left(D^{\mathrm{lat}}_{ k} + \gamma(W_k + M)\right)\right)
     = \mathrm{rank}\left(E_{>0}\left(D^{\mathrm{lat}}_{ k} + \gamma(W_k + 0.5k)\right)\right). 
\end{align*}
Applying Theorem \ref{thm_application} in the case $r = 1$ and $m = 0.5$, we get the result. 
\end{proof}

\section*{Acknowledgment}
The author is grateful to Hidenori Fukaya, Mikio Furuta, Shinichiroh Matsuo, Tetsuya Onogi and Satoshi Yamaguchi for introducing me to the problem in lattice gauge theory and for interesting discussions. 
She is also grateful to Yosuke Kubota for many helpful comments. 
This work is supported by Grant-in-Aid for JSPS KAKENHI Grant Number 20K14307. 

\bibliographystyle{plain}
\bibliography{lattice}
\end{document}